\documentclass{article}
\usepackage{amssymb}
\usepackage{tikz}
\usepackage{amsmath}
\usetikzlibrary{decorations.pathreplacing,snakes,mindmap}
\usetikzlibrary{arrows, calc, decorations.pathmorphing}
\usetikzlibrary{trees,positioning,arrows,shapes,snakes}

\newcommand{\shrink}{\kern -0.75\baselineskip}

\newtheorem{theorem}{Theorem}
\newtheorem{lemma}[theorem]{Lemma}
\newtheorem{corollary}[theorem]{Corollary}
\newtheorem{proposition}[theorem]{Proposition}

\newcommand{\proof}{{\it Proof.} \space}
\newcommand{\qed}{\ifhmode\unskip\nobreak\fi\ifmmode\ifinner
\else\hskip5 pt\fi\fi \hbox{\hskip5 pt
\vrule width4 pt  height6 pt  depth1.5 pt \hskip 1pt }}

\begin{document}
\title{Toughness and Hamiltonicity in \\ Random Apollonian Networks}

\author{Lilian Markenzon 
\footnote{Partially supported by grant 304706/2017-5, CNPq, Brazil.}
\\ 
       NCE -  Universidade Federal do Rio de Janeiro\\
                     markenzon@nce.ufrj.br
\and
            Christina F. E. M. Waga\\
            IME - Universidade do Estado do Rio de Janeiro   \\
            waga@ime.uerj.br           
}

\date{}

\maketitle

\thispagestyle{empty}


\begin{abstract}
In this paper we study the toughness of Random Apollonian Networks ({\em RAN}s),
a random graph model which generates planar graphs with power-law properties.
We  consider their important characteristics: 
 every  \emph{RAN} is a uniquely representable chordal graph and  
 a planar $3$-tree and as so, known results about  these classes can be particularized.
We establish a partition of the class in eight nontrivial subclasses 
and for each one of these subclasses we provide bounds for the toughness of their elements.
We also  study the hamiltonicity  of the elements of these subclasses.


 \end{abstract}


\section{Introduction}
\label{sec:introduction}

Over the last few years, the ever growing interest in social networks, the Web graph, biological
networks, etc., led  to  a great deal of research being built
around modelling real world networks.
In 2005, Andrade \emph{et al.} \cite{An05} introduced Apollonian networks (\emph{AN}s), inspired
by Apollonian packings \cite{KS43}, that  proved to
be an interesting tool for modeling real networked
systems. 
These networks can be produced as follows: start with a triangle
and then at each iteration, inside each triangle,
a vertex  is added and linked to the three vertices. 
Apollonian networks are scale-free, display the small-world effect 
and have a power-law degree distribution.
Generalizing {\em AN}s, the Random Apollonian Networks  (\emph{RAN}s)  were introduced by Zhou {\em et al.} \cite{ZYW05}; 
in this case, at each iteration of  a   {\em RAN}  a triangle
is randomly selected.
Some problems have been solved for these classes.
For instance, an exact analytical expression for the number of spanning trees in
{\em AN}s was achieved by Zhang {\em et al.} \cite{ZWC14};
the degree distribution, $k$ largest degrees and $k$ largest eigenvalues (for a fixed $k$) 
and the diameter of {\em RAN}s were studied in Frieze and Tsourakakis \cite{FT14};
Ebrahimzadeh {\em et al.} \cite{EFG13} follow  this line of research by studying the asymptotic properties of 
the longest  paths and presenting sharp estimates for the diameter of a {\em RAN}.
Others papers had employed a non-deterministic concept.

In this paper we focus in a different approach.
Considering  the equivalence between \emph{RAN}s and the planar $3$-trees ( i.e.,  
the maximal chordal planar graphs \cite{MJP06})
we analyse vulnerability properties of the networks, based on their clique-trees.

The toughness of a graph is an important invariant  
introduced in 1973 by  Chv\' atal \cite{Ch73} that deals with the vulnerability of a graph. 
Let the number of components of a graph $G=(V,E)$ be denoted  by $\omega(G)$. 
A graph $G$ is $t$-tough if
 $|S| \geq  t \,\omega(G-S)$ for every subset $S\subseteq V$  with  $\omega(G-S) > 1$. 
The \textit{toughness} of $G$, denoted $\tau(G)$, is the maximum value of $t$ for which $G$ is $t$-tough (taking $\tau(K_n) = \infty$,  $n \geq  1$). 
In other words, the toughness  relates the size of a separator with the number of components obtained after deleting it.  
It is important to highlight that the toughness can be directly related to the hamiltonicity of the graph.
Chv\' atal \cite{Ch73} has  established that every Hamiltonian graph is 1-tough, but   $1$-toughness does not ensure hamiltonicity. 
He  has also conjectured that there exists a $t$ such that every
 $t$-tough graph is Hamiltonian.
Some papers prove Chv\' atal's conjecture for different graph classes: 
$\tau(G)\geq 3/2$  for a split graph  \cite{KLM96}, 
$\tau(G) > 1$  for planar chordal graphs \cite{BHT99}, 
$ \tau(G) \geq 3/2$ for spider graphs \cite{KKS07} and
 $\tau(G)\geq 1$ for strictly chordal graphs \cite{MW16}. 
In particular for $k$-trees, 
Broersma \emph{et al.} \cite{BXY07} presented important results,
showing that if $G$ is a $k$-tree, $k \geq 2$, with toughness at least $(k+1)/3$, 
then $G$ is Hamiltonian.
For $k=2$, they prove that every $1$-tough $2$-tree on at least three vertices is Hamiltonian.
Kabela \cite{Ka19} has improved this result, showing that every $k$-tree (except for $K_2$) 
with toughness greather than $k/3$ is Hamiltonian.

In this paper we study the toughness of Random Apollonian Networks
based  on their characteristics: 
 every  \emph{RAN} is a uniquely representable chordal graph and, as so, it has a unique clique-tree;
every {\em RAN} is a planar $3$-tree and the results of B\"ohme  {\em et al.} \cite{BHT99} and
 Broersma \emph{et al.} \cite{BXY07} can be particularized.
We establish a partition of the class in eight nontrivial subclasses reliant on the structure of the clique-tree, 
and for each one of these subclasses we provide bounds for the toughness of their elements.
We also  study the hamiltonicity  of the elements of these subclasses.
Some well-known graphs, as the Goldner-Harary graph and the Nishizeki's example
of a non-Hamiltonian maximal planar graph \cite{Ni80}, fall in one of the defined subclasses.


\section{Background}
\label{sec:background}

Let $G  =(V,E)$,  be a connected graph, 
where $|V|=n$  and  $|E| = m$. 
The {\em set of neighbors\/} of a vertex $v \in V$ is denoted by
$N(v) = \{ w \in V; \{v,w\} \in E\}$. 
The \emph{degree} of a vertex $v\in V$ is $d(v)=|N(v) |$. 
For any $S \subseteq V$, 
the subgraph of $G$ induced by $S$ is denoted $G[S]$. 
 If $G[S]$ is a complete graph then $S$ is a \emph{clique} in $G$. 
A vertex $v\in V$ is said to be {\em
simplicial\/} in $G$ when $N(v)$ is a clique in $G$.
The set of simplicial vertices of $G$  is denoted by $SI$.

The graphs  $G=(V,E)$ and $G'=(V',E')$ are isomorphic if there is a bijective function $f:V \to V'$ 
such that for all $v,w  \in V,$ $\{v,w \} \in E$ if and only if 
$\{f(v),f(w) \} \in E'$, i.e, $f$ preserves adjacency.

Basic concepts about chordal graphs are assumed to be known and 
can be found in  Blair and Peyton \cite{BP93}  and Golumbic \cite{Go04}.
In this section, the most pertinent concepts are reviewed.

A subset $S \subset V$ is
a {\em separator} of $G$ if at least two vertices in the same connected
component of $G$ are in two distinct connected components of
$G[V\setminus S]$. 

Let $G = (V, E)$ be a chordal graph and $u,v  \in V$. 
A subset $S \subset V$  is a {\em vertex separator}  for
non-adjacent vertices $u$  and $v$  (a $uv$-separator) if the
removal of $S$ from the graph separates $u$ and $v$  into distinct
connected components. 
If no proper subset of $S$  is a $uv$-separator then $S$ is a {\em minimal $uv$-separator}. 
When the pair of vertices remains unspecified, we refer to $S$  as a {\em
minimal vertex separator} ({\em mvs}). 
The set of minimal vertex separators is denoted by $\mathbb S$.

The {\it clique-intersection graph\/} of a graph $G$ is the
connected weighted graph whose vertices are the maximal cliques of $G$ and whose
edges connect vertices corresponding to non-disjoint maximal cliques.
Each edge is assigned an integer weight, given by the cardinality of the
intersection between the maximal cliques represented by its endpoints.
Every maximum-weight spanning tree of the clique-intersection graph of $G$
is called a {\it clique-tree\/} of $G$.
The set of maximal cliques of $G$ is denoted by $\mathbb{Q}$.
A clique-tree of $G$ represents the graph $G$.
Clique-trees satisfy the  \emph{induced subtree property} (\emph{ISP}):
  ${\mathbb Q}(v)$ induces a subtree of  the clique-tree $T$ of $G$ where  
  ${\mathbb Q}(v)$ is the set of maximal cliques containing the vertex $v\in V$.  
 Observe that  each maximal clique $Q\in \mathbb{Q}$ is related to a vertex $q$ of the clique-tree $T$ of $G$.
A {\em simplicial clique} is a maximal clique containing at least one simplicial vertex.

For a chordal graph $G$ 
and a clique-tree $T$ of $G$,
a set $S\subset V$ is a \emph{mvs}  of $G$ if
and only if $S= Q\cap Q' $ for some edge $\{Q, Q'\}$ in $T$. 
Moreover, the multiset  ${\mathbb M}$ of
the minimal vertex separators of $G$ is the same for every
clique-tree of $G$.
The {\em multiplicity} of the minimal vertex separator $S$, denoted by $ \mu (S)$,   is the number of times that $S$ appears in  ${\mathbb M}$. 
The determination of the minimal vertex separators and their multiplicities 
can be performed in linear time \cite{MP10}.   

A \emph{$k$-regular tree} is a tree in which every vertex that  is not a leaf has degree $k$.


\section{Some subclasses of chordal graphs}

In this paper we deal with some subclasses of chordal graphs which are now reviewed.

A chordal graph is called a {\em uniquely representable chordal graph} \cite{KM02}
(briefly\emph{ ur-chordal graph})  if it has exactly one clique-tree.

\begin{theorem}{\bf \cite{KM02}}\label{theo:caract-urchor} 
Let  $G$ be a chordal graph.  Then, $G$ is uniquely representable if and only if 
 there is no
proper containment between any minimal vertex separators and 
all  minimal vertex separators are   of  multiplicity
one. 
\end{theorem}

A {\em $k$-tree}, $k > 0$, firstly presented in \cite{Ro74}, can be inductively defined as follows:
\begin{enumerate}
  \item Every complete graph with $k+1$ vertices is a $k$-tree.
  \item  If $G = (V, E)$ is a $k$-tree, $v \notin V$ and $S \subseteq V$
      is a $k$-clique of $G$, then $G' = (V \cup \{v\}, E \cup \{ \{v,w\} \mid w \in S \})$
      is also a $k$-tree.
  \item Nothing else is a $k$-tree.
\end{enumerate}

Two subclasses  of $k$-trees are  the  simple-clique $k$-trees (SC $k$-trees)  and 
the $k$-path graphs  \cite{MJP06}.
A {\em SC $k$-tree}, $k>0$,  is a uniquely representable $k$-tree.
A complete graph on $k+1$ vertices is a {\em $k$-path graph}, $k > 0$; 
if $n > k+1$,  $G$ is a {\em $k$-path graph} if
and only if $G$ has exactly two simplicial vertices.


\subsection{Apollonian networks}

Several results can be deduced from the fact that 
Random Apollonian Networks  are the same as  SC 3-trees, 
proved to be the maximal chordal planar graphs
by  Markenzon {\em et al.} \cite{MJP06}. 

Consider  $G=(V,E)$   a \emph{RAN} on $n$ vertices. 
Since it is a $3$-tree, it is immediate that 
 every maximal clique has cardinality 4 and 
every minimal vertex separator  has cardinality 3. 
Graph $G$ has $n-3$ maximal cliques and, since it is uniquely representable,
every set of three distinct vertices appears at most in two maximal cliques;
for $n \geq 5$, the number of simplicial vertices is less or equal the number of non-simplicial ones.

\begin{proposition} Let $G = (V,E)$ be a non-complete RAN and $T = (V_T, E_T )$ be its clique-tree. 
\begin{enumerate}
  \item $|V_T|= |\mathbb{Q}|=n-3$. 
  \item $|E_T|= |\mathbb{S}|=n-4$.
  \item The number of leaves in  $T$ is the number of simplicial vertices in $G$. 
  \item Internal vertices of $T$ contain exclusively vertices which belong to minimal vertex separators.
   \item  Every vertex of $T$ has  degree less or equal 4. 

\end{enumerate}
\end{proposition}

%
\section{Toughness} 

Chv\'atal \cite{Ch73} had  introduced toughness in 1973. 
Let $\omega(G)$ denote the number of components of a graph $G=(V,E)$. 
A graph $G$ is $t$-tough if $|S| \geq  t \,\omega(G -S)$ for every subset $S\subseteq V$  with  $\omega(G - S) > 1$. 
The \textit{toughness} of $G$, denoted $\tau(G)$, is the maximum value of $t$ 
for which $G$ is $t$-tough (taking $\tau(K_n) = \infty$ for all $n \geq  1$). 
Hence if $G$ is not complete, $\tau(G)= min\big\{ \frac{|S|}{\omega(G-S)}\big\}$, 
where the minimum is taken over all separators  $S$ of vertices in $G$ \cite{B06}.

We present below the most important known results directly related to our paper.

\begin{theorem}\label{theo:Chavtal0} {\rm\cite{Ch73} }
If $H$ is a spanning subgraph of $G$  then $\tau(H)\leq \tau(G)$. 
\end{theorem}

\begin{theorem}\label{theo:Chavtal} {\rm\cite{Ch73} }
If $G$ is Hamiltonian then $\tau(G)\geq 1$. 
\end{theorem}

\begin{theorem}\label{theo:Bohme} {\rm\cite{BHT99} }
Let $G$ be a planar chordal graph with $\tau(G) >1$. Then $G$ is Hamiltonian. 
\end{theorem}

\begin{theorem}\label{theo:ktreehamileq} {\rm\cite{BXY07} }
Let $G\neq K_2$ be a $k$-tree. Then $G$ is Hamiltonian if and only if $G$ contains a $1$-tough spanning $2$-tree.\end{theorem}

\begin{theorem}\label{theo:ktreehamilsuf} {\rm\cite{BXY07}  }
If $G\neq K_2$ is a $\frac{k+1}{3}$-tough $k$-tree, $k\geq 2$, then  $G$ is Hamiltonian. 
\end{theorem}

\begin{lemma} {\rm\cite{BXY07}}  \label{lema:simp} 
Let $G\neq K_k$ be a $k$-tree ($k\geq 2$). 
Then $\tau(G-\{v\}) \geq \tau(G)$ for all simplicial vertex $v$ of $G$.
\end{lemma}

\begin{corollary}
Let $G\neq K_k$ be a $k$-tree ($k\geq 2$) and $SI$ be the set of simplicial vertices of $G$. Then $\tau(G- SI) \geq \tau(G)$.
\end{corollary} 

\proof  
Consider $SI=\{v_1,\dots,v_s \}$ and the subgraphs $G_1=G-\{v_1\}, 
G_2=G_1-\{v_2\}, \dots, G_s=G_{s-1}-\{v_s\}$ of $G$.
By Lemma \ref{lema:simp},   $\tau(G -  SI) =\tau(G_s) \geq \dots \geq \tau(G_1) \geq \tau(G)$.    \qed \\

%

\section{Clique-tree related subclasses of \emph{RAN}s}

In this section, several subclasses of \emph{RAN}s are defined, based on the structure of its unique clique-tree.
This approach  will allow us to present a detailed analysis of the toughness (and hamiltonicity)  of \emph{RAN}s.

Let  $G$ be a \emph{RAN} and $q_i$ and $q_j$  be two vertices of degree $4$ of the clique-tree $T$ of $G$. 
Let $P_{i,j}=\langle q_i, q_1,  \dots, q_p, q_j \rangle$  be the path joining $q_i$ and $q_j$ in $T$ 
such that $q_i$ and $q_j$ are adjacent or the degree of all vertices $q_k, 1\leq k \leq p$, is less than or equal to  $3$.
This path is called a {\em neat path} of $G$. 

Let $P_{i,j}$ be a neat path of $T$.
Consider the internal vertices of $P_{i,j}$,  $P=\langle q_1, q_2, \ldots, q_p \rangle$.
If  $P$ is empty or all the vertices of $P$ have degree $3$  it is called a {\em fat path}.
Otherwise it  is called a {\em slim path};  it has at least one vertex of degree $2$ and  $p\geq 1$. 

All graphs considered for now on are non-complete graphs. 
The smallest non-complete \emph{RAN}  has 5 vertices  and,  up to isomorphism, establish a unitary class $C_0$. 


Let  $G=(V,E)$ be a \emph{RAN} on $n\geq 6$ vertices  and $T=(V_T, E_T)$ its clique-tree.
\begin{itemize}

 \item $G$  belongs to $C_1$ if  $T$ is a  4-regular tree.
  
    $G$ has $n=8+3\ell$ vertices, $\ell \geq 0$,  and $|SI|=4+2\ell$. 
    
\item $G$  belongs to $C_2$ if  $T$ is a  3-regular tree.

$G$ has  $n=7+2\ell $ vertices, $\ell\geq 0$, and $|SI|=\lfloor \frac{n}{2}\rfloor = \frac{6+2\ell }{2} = 3+\ell$.

\item $G$  belongs to $C_3$ if  $T$ is a  2-regular tree.
   
 $G$ has $n\geq 6$ and $|SI|=2$. Furthermore, $G$  is a  3-path graph. 
 
\item $G$    belongs to $C_4$ if    $T$ is not $k$-regular and it has no vertices of degree $4$.  

$G$ has  $n\geq 8$ vertices  and $|SI|\geq 3$. 

\item $G$     belongs to $C_5$ if $T$ is not $k$-regular and  it has exactly one vertex of degree $4$. 

$G$ has   $n\geq 9$ vertices and $|SI|\geq 4$.

\item $G$     belongs to $C_6$   if  $T$ is not $k$-regular and it has at least one fat path.

$G$ has   $n\geq 12$ vertices and $|SI|\geq 6$.

\item $G$     belongs to $C_7$   if   
$T$ is not $k$-regular, it has no fat paths  and it  has at least 
a neat path $P_{i,j} =  \langle q_i, P, q_j \rangle$ with one of the following properties:
\begin{itemize}
\item $|Q_i \cap Q_j|=2$ with $p \geq 2$ or 

\item $|Q_i \cap Q_j|=  1$ with $p \geq 3$ and $G$ contains at least one maximal clique $Q_k$ 
such that $(Q_i \cup Q_j) \supset Q_k$, $d(q_k) =3$ or

\item $|Q_i \cap Q_j|= 0$ with $p \geq 4$ and $G$ contains at least two maximal cliques $Q_k$ and $Q_\ell$ 
such that $(Q_i \cup Q_j) \supset Q_k$, $(Q_i \cup Q_j) \supset Q_\ell $, $d(q_k) =d(q_\ell) = 3$.

\end{itemize}

$G$ has  $n\geq 13$ vertices and $|SI|\geq 6$.

\item $G$    belongs to $C_8$ if  $G$  does not belong to any one of the classes defined above.  

$G$ has  $n\geq 12$ vertices, $|SI|\geq 6$.  

\end{itemize}

It is important to note that classes $C_7$ and $C_8$ encompass all the {\em RAN}s that have only slim paths.
The following result is immediate.

\begin{theorem}
Classes $C_0, C_1, C_2, C_3, C_4, C_5,  C_6,C_7$ and $C_8$  establish a partition of the
non-complete  Random Apollonian Networks.  
\end{theorem}
  
Some  observations about non-isomorphic \emph{RAN}s and their clique-trees can be stated. 
Graphs  with the same number of vertices
can belong to different classes or to the same class $C_i$ and their clique-trees can be isomorphic or not, since 
the isomorphism depends only on the structure of the tree.
The graphs $G_1$, $G_2$, $G_3$ and $G_4$, depicted in Figure  \ref{fig:noniso},
are all non-isomorphic \emph{RAN}s.
Graphs $G_1$, $G_2$  and $G_3$ belong to $C_7$;
$G_1$ and $G_2$ have isomorphic clique-trees and $G_1$ and $G_3$  do not.
Graph $G_4$ belongs to $C_8$; $G_1$ and $G_4$ have also isomorphic clique-trees.

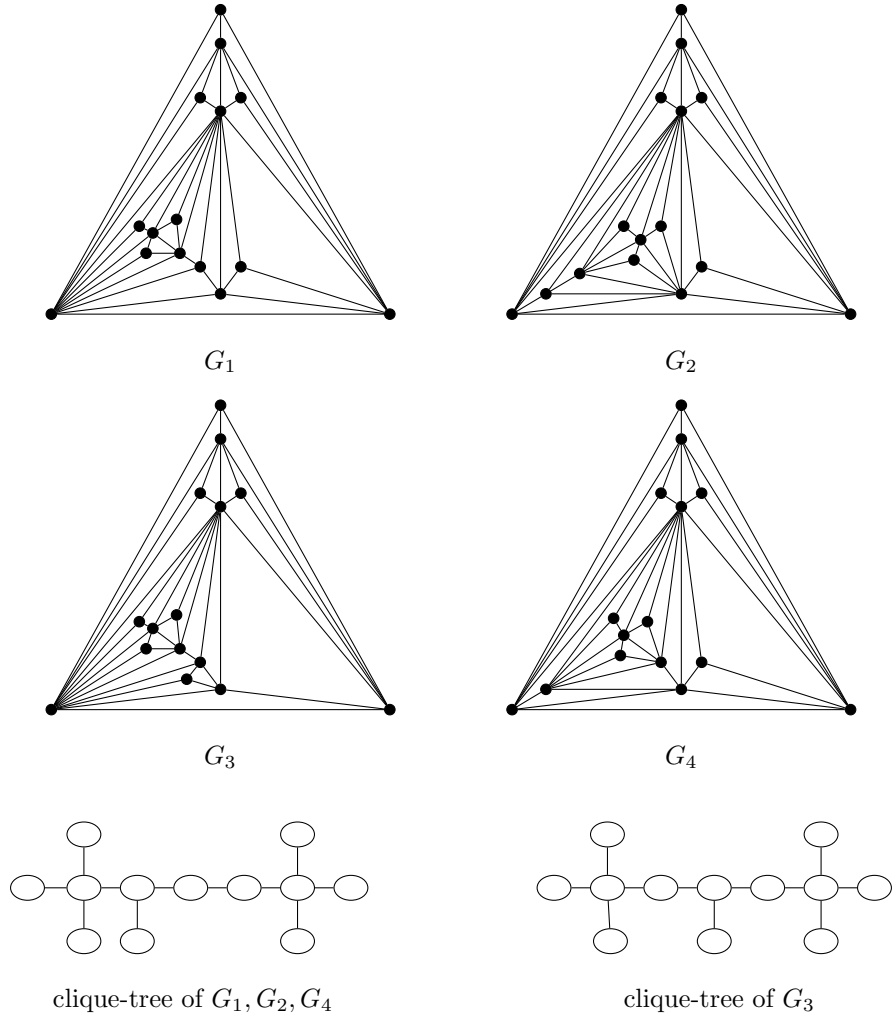
\begin{figure}[!h] \begin{center}
\tikzset{
    graphnode/.style={
      draw,fill,
      circle,
      minimum size=1.4mm,
      inner sep=0,
    }
}
\begin{minipage}{6cm}
\begin{center}
\begin{tikzpicture}
[every node/.style={graphnode},scale=.9]

   \node  (b) at (-0.5,0) {};
  \node   (c) at (4.5,0) {};
  \node  (4) at (2,4.5) {};
 
  \node   (d) at (2,4) {};
  \node   (a) at (2,3) {}; 
   \node  (1) at (2,0.3) {};
   
 \node  (2) at (1.7,3.2) {};
 \node  (3) at (2.3,3.2) {};

 \node  (5) at (1.7,0.7) {};
 \node  (5') at (2.3,0.7) {};

\node  (6) at (1.4,0.9) {};
\node  (7) at (1,1.2) {};

\node  (8) at (0.8,1.3) {};
\node  (9) at (0.9,0.9) {};
\node  (10) at (1.35,1.4) {};
 \foreach \from/\to in {a/b, a/c,a/d,b/c,b/d,c/d,4/b,4/c,4/d,1/a,1/b,1/c,2/d,2/a,2/b,3/a,3/c,3/d,5'/1,5'/a,5'/c,5/b,5/1,5/a,6/5,6/b,6/a,7/b,7/a,7/6,8/7,8/b,8/a,9/b,9/7,9/6,10/a,10/7,10/6}
   \draw (\from) -- (\to);
\node at (2,-0.7) [draw=none,fill=none] {$G_1$};
\end{tikzpicture}
\end{center}
\end{minipage}\hfill
\begin{minipage}{6cm}
\begin{center}
\begin{tikzpicture}[every node/.style={graphnode},scale=.9]

\node  (b) at (-0.5,0) {};
  \node   (c) at (4.5,0) {};
  \node  (4) at (2,4.5) {};
 
  \node   (d) at (2,4) {};
  \node   (a) at (2,3) {}; 
   \node  (1) at (2,0.3) {};
   
 \node  (2) at (1.7,3.2) {};
 \node  (3) at (2.3,3.2) {};

 \node  (5) at (0,0.3) {};
 \node  (5') at (2.3,0.7) {};

\node  (6) at (0.5,0.6) {};
\node  (7) at (1.4,1.1) {};

\node  (8) at (1.15,1.3) {};
\node  (9) at (1.7,1.3) {};
\node  (10) at (1.3,0.8) {};
 \foreach \from/\to in {a/b, a/c,a/d,b/c,b/d,c/d,4/b,4/c,4/d,1/a,1/b,1/c,2/d,2/a,2/b,3/a,3/c,3/d,5'/1,5'/a,5'/c,5/b,5/1,5/a,6/5,6/1,6/a,7/1,7/a,7/6,8/7,8/6,8/a,9/a,9/7,9/1,10/1,10/7,10/6}
   \draw (\from) -- (\to);
    \node at (2,-0.7) [draw=none,fill=none] {$G_2$};
\end{tikzpicture}
\end{center}
\end{minipage}\\

\medskip

\begin{minipage}{6cm}
\begin{center}
\begin{tikzpicture}[every node/.style={graphnode},scale=.9]
   \node  (b) at (-0.5,0) {};
  \node   (c) at (4.5,0) {};
  \node  (4) at (2,4.5) {};
 
  \node   (d) at (2,4) {};
  \node   (a) at (2,3) {}; 
   \node  (1) at (2,0.3) {};
   
 \node  (2) at (1.7,3.2) {};
 \node  (3) at (2.3,3.2) {};

 \node  (5) at (1.7,0.7) {};
 \node  (5') at (1.5,0.45) {};

\node  (6) at (1.4,0.9) {};
\node  (7) at (1,1.2) {};

\node  (8) at (0.8,1.3) {};
\node  (9) at (0.9,0.9) {};
\node  (10) at (1.35,1.4) {};
 \foreach \from/\to in {a/b, a/c,a/d,b/c,b/d,c/d,4/b,4/c,4/d,1/a,1/b,1/c,2/d,2/a,2/b,3/a,3/c,3/d,5'/1,5'/5,5'/b,5/b,5/1,5/a,6/5,6/b,6/a,7/b,7/a,7/6,8/7,8/b,8/a,9/b,9/7,9/6,10/a,10/7,10/6}
   \draw (\from) -- (\to);
\node at (2,-0.7) [draw=none,fill=none] {$G_3$};
\end{tikzpicture}
\end{center}
\end{minipage}\hfill
\begin{minipage}{6cm}
\begin{center}
\begin{tikzpicture}[every node/.style={graphnode},scale=.9]
\node  (b) at (-0.5,0) {};
  \node   (c) at (4.5,0) {};
  \node  (4) at (2,4.5) {};
 
  \node   (d) at (2,4) {};
  \node   (a) at (2,3) {}; 
   \node  (1) at (2,0.3) {};
   
 \node  (2) at (1.7,3.2) {};
 \node  (3) at (2.3,3.2) {};

 \node  (5) at (0,0.3) {};
 \node  (5') at (2.3,0.7) {};

\node  (6) at (1.7,0.7) {};
\node  (7) at (1.15,1.1) {};

\node  (8) at (1,1.35) {};
\node  (9) at (1.5,1.3) {};
\node  (10) at (1.1,0.8) {};
 \foreach \from/\to in {a/b, a/c,a/d,b/c,b/d,c/d,4/b,4/c,4/d,1/a,1/b,1/c,2/d,2/a,2/b,3/a,3/c,3/d,5'/1,5'/a,5'/c,5/b,5/1,5/a,6/5,6/1,6/a,7/5,7/a,7/6,8/7,8/5,8/a,9/a,9/7,9/6,10/6,10/7,10/5}
   \draw (\from) -- (\to);
    \node at (2,-0.7) [draw=none,fill=none] {$G_4$};
\end{tikzpicture}
\end{center}
\end{minipage}\\

\medskip

\begin{center}
\begin{tikzpicture}
  \tikzstyle{vert}=[ellipse,draw,fill=none,minimum width=4.5mm, minimum height=1mm]
      
         \node  [vert]  (1) at (-0.5,0) {};
      \node [vert]  (2) at (0.25,0) {};
    \node [vert]  (3) at (0.96,0) {};
   \node  [vert]  (4) at (1.67,0) {};
  \node  [vert]  (5) at (2.38,0) {}; 
   \node [vert]   (6) at (3.09,0) {};
\node [vert]   (7) at (3.8,0) {};
\node [vert]   (8) at (0.25,0.71) {};
\node [vert]   (9) at (3.09,0.71) {};
\node [vert]   (10) at (0.25,-0.71) {};
\node [vert]    (11) at (0.96,-0.71) {};
\node [vert]   (12) at (3.09,-0.71) {};
 \foreach \from/\to in {1/2,2/3,3/4,4/5,5/6,6/7,2/8,2/10,3/11,6/9,6/12}
   \draw (\from) -- (\to);
\node at (1.7,-1.5) [draw=none,fill=none] {clique-tree  of $G_1,G_2,G_4$};

  \node  [vert]  (a1) at (6.5,0) {};
      \node  [vert]  (a2) at (7.21,0) {};
    \node  [vert]   (a3) at (7.92,0) {};
   \node   [vert]  (a4) at (8.63,0) {};
  \node  [vert]   (a5) at (9.34,0) {}; 
   \node  [vert]  (a6) at (10.05,0) {};
  \node   [vert]  (a7) at (10.76,0) {};
  
   \node  [vert]   (a8) at (7.21,0.71) {};
\node   [vert] (a9) at (10.05,0.71) {};

\node  [vert]  (a10) at (7.25,-0.71) {};
\node  [vert]  (a11) at (8.63,-0.71) {};
\node  [vert]  (a12) at (10.05,-0.71) {};
 \foreach \from/\to in {a1/a2,a2/a3,a3/a4,a4/a5,a5/a6,a6/a7,a2/a8,a2/a10,a4/a11,a6/a9,a6/a12}
   \draw (\from) -- (\to);
    \node at (8.7,-1.5) [draw=none,fill=none] {clique-tree of  $G_3$};
\end{tikzpicture}
\end{center}

\vspace{-0.7cm}
\caption{Non-isomorphic \emph{RAN}s}
\label{fig:noniso}\end{center}
\end{figure}


\section{Main results -- toughness}

In this section,  results on the toughness of the subclasses defined in Section 5 are presented. 

\begin{theorem}\label{theo:toug-C0}  Let $G \in C_0$. Then 
$\tau(G_0)=\frac{3}{2}$. 
\end{theorem}
\begin{proof}  Immediate.   \qed \\
\end{proof}

We consider $G=(V,E)$  a \emph{RAN}  on $n \geq  6$ vertices from Theorem \ref{theo:toug-C1} through
Theorem \ref{theo:toug-C7}. 

\begin{theorem}\label{theo:toug-C1}  Let $G\in C_1$. 
Then 
$\tau(G)=\frac{n+4}{2n-4}$. 
\end{theorem}
\begin{proof}  
If  $n=8$, trivially,    $\tau(G)=1$.
Otherwise,  $n=8+3\ell$, $\ell  \geq 1$.
Consider the removal of the set $\bigcup_{S\in \mathbb{S}} S$ of non-simplicial vertices. We obtain a disconnected graph with $|SI|$ components and the value
$$\frac{n-|SI|}{|SI|}=\frac{8+3\ell -(4+2\ell)}{4+2\ell} = \frac{4+\ell}{4+2\ell} =\frac{n+4}{2n-4}.$$

Let us now consider a new separator consisting of all elements of $\bigcup_{S\in \mathbb{S}} S$ 
except  one, a  non-simplicial vertex $v$. 
By the induced subtree property, the set ${\mathbb Q}(v)$  of maximal cliques containing the vertex $v$
 induces a subtree of  the clique-tree $T$ of $G$.
Let   $T_v$ be this subtree; $T_v$ is a  $3$-regular tree. 
So, vertex  $v$ belongs to at least three simplicial cliques. 
 
Consider the set $A =  \bigcup_{S\in \mathbb{S}} S\setminus \{ v\}$. 
The graph    $G \left[ V\setminus A\right] $ has fewer components than  the graph $G \left[ V\setminus \bigcup_{S\in \mathbb{S}} S \right] $ because the adjacencies of  vertex $v$ are kept (at least 3 simplicial vertices). 
So the number of components of  $G \left[ V\setminus A\right] $ is $4+2\ell -c+1=5+2\ell-c$, $c\geq 3$.  
As $|A|=n-|SI|-1$,    $  \frac{4+\ell}{4+2\ell}  < \frac{3+\ell}{5+2\ell-c}$.
Then  $\tau(G)=\frac{n+4}{2n-4}$.  \qed \\
\end{proof}

\begin{corollary} The Goldner-Harary graph belongs to $C_1$. \end{corollary}

\begin{theorem}\label{theo:toug-C2} 
Let  $G\in C_2$. 
Then 
$\tau(G)=\frac{n+1}{n-1}$. 
\end{theorem}
\begin{proof} 
If $n=7$,  $\tau(G)=\frac{4}{3}$.
Otherwise, $n=7+2\ell$, $\ell\geq 1$. 
Consider the removal of the set $\bigcup_{S\in \mathbb{S}} S$ of non-simplicial vertices.
We obtain a disconnected graph with $|SI|$ components and  the value 
$$\frac{n-|SI|}{|SI|}=\frac{ 7+2\ell-(3+\ell)} {3+\ell} =\frac{ 4+\ell} {3+\ell}=\frac{n+1}{n-1}.$$

Let us consider another  separator consisting of all elements of $\bigcup_{S\in \mathbb{S}} S$ except  one, a  non-simplicial vertex $v$. 
Consider the set $A =  \bigcup_{S\in \mathbb{S}} S\setminus \{ v\}$ such that  $|A|=n-|SI|-1=3+\ell$, and the graph $G'=G \left[ V\setminus A\right]$. 

As $G \in C_2$, it is possible that there is one universal vertex. If $v$ is this universal vertex, $G'$ is a connected graph and $   \frac{3+\ell}{1}>\frac{4+\ell}{3+\ell }$. 
Otherwise,  $v$ is not a universal vertex, let  $T_v$ be the subtree of  the clique-tree $T$ 
of $G$ induced by the set ${\mathbb Q}(v)$ (induced subtree property). 
We know that every tree with at least two vertices has at least two leaves. 
In our case,  these leaves are simplicial cliques of $G$,  i.e.,  $v$ belongs to at least two simplicial cliques.
So, $G'$ has $3+\ell -c+1= 4+\ell-c$ components, $c \geq 2$,  and $\frac{3+\ell}{4+\ell-c}> \frac{4+\ell}{3+\ell}$. 
Then $\tau(G)=\frac{n+1}{n-1}$.  \qed \\
\end{proof}

In \cite{MW14}, bounds to the toughness of $k$-path graphs, $k\geq 2$, were presented. 
Hence, we can present the following result. 

\begin{theorem}\label{theo:toug-C3}
Let  $G\in C_3$. 
Then 
$\left\{ \begin{array} {ll}
\frac{n}{n-2}   \leq \tau(G) \leq \frac{3}{2}   &   \textrm{ if } n \textrm{ is even} \\\\
\frac{n+1}{n-1}\leq \tau(G) \leq \frac{3}{2}  &  \textrm{ if } n \textrm{ is odd.}\,\,\,\end{array}  \right.$
\end{theorem}

The equalities of the bound values are  achieved by graphs of two subclasses of $k$-path graphs: $k$-ribbon and $k$-fan graphs.

\begin{theorem}\label{theo:toug-C4} Let  $G\in C_4$.
Then 
$\frac{n+2}{n} \leq \tau(G) \leq  \frac{4}{3}$.
\end{theorem}

\begin{proof}  
If $n=8$, $\tau(G) = \frac{4}{3}$.
Otherwise,  consider the  clique-tree $T$ of $G$ and a tree $T'$ obtained from $T$ 
by the addition of  one leaf to every vertex of degree 2. 
So, $T'$ is a clique-tree of some graph $G'\in C_2$ on $n+\ell$ vertices, $\ell\geq 1$,  
and $ \tau(G')=\frac{n+\ell+1}{n+\ell -1}$, by Theorem \ref{theo:toug-C2}. 
For each new  leaf $q$ of $ T'$, there is a  maximal clique $Q$ in $G'$ with a new simplicial vertex. 
By Lemma \ref{lema:simp},   $ \tau(G) \geq   \tau(G')$. 
If $\ell=1$, $\frac{n+2}{n}$ $ \leq  \tau(G)$.
Furthermore, $n\geq9$ and $\frac{n+2}{n} \leq  \frac{4}{3}$. 
 \qed \\
\end{proof}

\begin{theorem}\label{theo:toug-C5}
Let $G\in C_5$. Then $\tau(G) = 1$.
\end{theorem}
\begin{proof}
Let $T=(V_T, E_T)$ be the clique-tree of $G$.
There is one vertex $q\in V_T$ such that $d(q)=4$.
So, the removal of the vertices of the clique $Q$ from $G$ entails four remaining connected components.
Then  $\tau(G) \leq 1$.

We are going to prove that $G$ has a Hamiltonian cycle, showing that,
as view in Theorem \ref{theo:ktreehamileq}, it contains a $1$-tough spanning $2$-tree, i.e., a maximal
outerplanar graph (a {\em mop}).
Equivalently a {\em mop} is a SC $2$-tree.
In order to obtain this result we are going to rebuild $T$.

Let us consider a subtree $T'$ of $T$ (associated  with a subgraph $G'$ of $G$)  
containing  vertex $q$, being $Q=\{a,b,c,d\}$, and its four adjacent vertices.
Graph $G'$ has four simplicial vertices $v_1, v_2, v_3$ and $v_4$.
It is immediate that there is a \emph{mop} $M$ that is a spanning subgraph of $G'$;
without loss of generality, a Hamiltonian cycle of $G'$  is $\langle a,v_1,b,v_2,c, v_3, d,v_4,a \rangle$.
So, each simplicial vertex is adjacent to two non-simplicial  ones.

The remaining of the clique-tree $T$ will be built with some restrictions.
The addition of new vertices to $T'$ 
(corresponding to  maximal cliques of $G$, each one with one new vertex of $G'$)
will be performed only on leaves of the tree $T'$.
In each iteration one or two new vertices will be added to $T'$, since 
each leaf $q_i$ of $T'$ can have one or two new adjacent vertices, by the definition of $C_5$.

\medskip

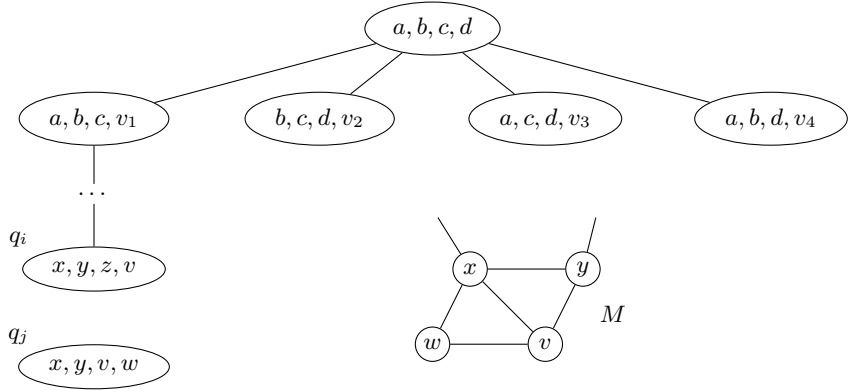
\begin{figure}[h]
\centering
{\small
\begin{tikzpicture} 
 \tikzstyle{block}= [ellipse,draw,minimum width=40pt]
  \tikzstyle{vert}=[draw,fill=white,circle,    inner sep=0pt, minimum width=13pt,align=center]
    \node [block] (0)  at (4.5,1.2) {$a,b,c,d$};
    \node[block] (1)     at (0,0) {$a,b,c,v_1$};
     \node[block] (2)     at (3,0) {$b,c,d,v_2$};
     \node[block] (3)     at (6,0) {$a,c,d,v_3$};
   \node[block] (4)     at (9,0) {$a,b,d,v_4$};
    \node (x)  at (0,-1) [draw=none,fill=none] {$\dots$};
     \node[block] (5)  at (0,-2) {$x,y,z,v$};
   \node[block] (6)  at (0,-3.3)  {$x,y,v,w$};
   \node  at (-1,-1.6) [draw=none,fill=none] {$q_i$};     
    \node  at (-1,-2.9) [draw=none,fill=none] {$q_j$}; 
    \draw (0)--(1);  \draw (0)--(2);   \draw (0)--(3);   \draw (0)--(4);  \draw (1)--(x);  \draw (x)--(5);  

 \node (a0)  at (4.5,-1.2) [draw=none,fill=none] {$$};
  \node (b0)  at (6.7,-1.2) [draw=none,fill=none] {$$};
  \node[vert] (A)  at (5,-2) {$x$}; 
  \node[vert] (B)  at (6.5,-2) {$y$};  
    \node[vert] (C)  at (4.5,-3) {$w$}; 
     \node[vert] (D)  at (6,-3) {$v$};
 \draw (a0)--(A) --(D) --(B) --(A); 
  \draw (b0)--(B); 
 \draw (A) -- (C) -- (D); 
   \node  at (6.9,-2.6) [draw=none,fill=none] {$M$};     

  \end{tikzpicture}
}
\caption{Case  1 of Theorem \ref{theo:toug-C5}} 
\label{fig:teoc5a}
\end{figure}

 Without loss of generality, see Figure \ref{fig:teoc5a}. 
 Leaf $q_i$  has one simplicial vertex $v$,  that is adjacent in the \emph{mop} $M$ to vertices $x$ and $y$.
If only one vertex $q_j$ is added to $T'$, the addition of the vertex 
$w \in Q_j= \{x,y,v,w\}$  to the \emph{mop} $M$ is immediate.
It is adjacent to $v$ and $x$ or $v$ and $y$ in $M$ since vertex $v$ is mandatory in the \emph{mop} by   definition.

If  two adjacent vertices $q_j$ and $q_k$ are added to $T'$ (see Figure \ref{fig:teoc5b}), they must be analyzed together.
Let us suppose the following situation:

\begin{itemize}
  \item leaf $q_i$ of $T'$ corresponds to clique $Q_i=\{x, y, z, v\}$; vertex $v$ is adjacent to $x$ and $y$ in $M$;
  \item adjacent vertices to be added: cliques $q_j$ (new vertex $w \in Q_j$) and $q_k$ (new vertex $t\in Q_k$).
\end{itemize}

Observe that $Q_i \cap Q_j$ can be $\{x, y, v\}$, $\{ x, z, v\} $ or
$\{ y, z, v\}$; $Q_i \cap Q_k$ have the same number of  choices but they are different. 
Hence it is always possible to add vertices $w$ and $t$ to the \emph{mop} $M$;
two leaves are considered in the clique-tree $T'$ and vertex $q_i$ is no long a leaf.

In both cases, graph $G$ has a Hamiltonian cycle  and $\tau(G)=1$.
\qed       

\medskip

\begin{figure}[h]
\centering
{\small
\begin{tikzpicture} 
 \tikzstyle{block}= [ellipse,draw,minimum width=40pt]
  \tikzstyle{vert}=[draw,fill=white,circle,    inner sep=0pt, minimum width=13pt,align=center]
    \node [block] (0)  at (4.5,1.2) {$a,b,c,d$};
    \node[block] (1)     at (0,0) {$a,b,c,v_1$};
     \node[block] (2)     at (3,0) {$b,c,d,v_2$};
       \node[block] (3)     at (6,0) {$a,c,d,v_3$};
   \node[block] (4)     at (9,0) {$a,b,d,v_4$};
    \node (x)  at (0,-1) [draw=none,fill=none] {$\dots$};
      \node[block] (5)  at (0,-2) {$x,y,z,v$};
      \node[block] (6)  at (-0.8,-3)  {$x,z,v,w$};
      
        \node[block] (7)  at (1.2,-3)  {$y,z,v,t$};
      \node  at (-1,-1.6) [draw=none,fill=none] {$q_i$};     
      \node  at (-1.7,-2.6) [draw=none,fill=none] {$q_j$}; 
       \node  at (2.1,-2.6) [draw=none,fill=none] {$q_k$}; 
       \draw (0)--(1);  \draw (0)--(2);   \draw (0)--(3);   \draw (0)--(4);  \draw (1)--(x);  \draw (x)--(5); 

 \node (a0)  at (4.5,-1.2) [draw=none,fill=none] {$$};
  \node (b0)  at (6.7,-1.2) [draw=none,fill=none] {$$};
  \node[vert] (A)  at (5,-2) {$x$}; 
  \node[vert] (B)  at (6.5,-2) {$y$};  
    \node[vert] (C)  at (4.5,-3) {$w$}; 
     \node[vert] (D)  at (6,-3) {$v$}; 
       \node[vert] (E)  at (7.5,-3) {$t$};
 \draw (a0)--(A) --(D) --(B) --(A); 
  \draw (b0)--(B); 
 \draw (A) -- (C) -- (D) --(E) --(B); 
   \node  at (7.5,-2.3) [draw=none,fill=none] {$M$};     
 \end{tikzpicture}
}
\caption{Case  2 of Theorem \ref{theo:toug-C5}} 
\label{fig:teoc5b}
\end{figure}
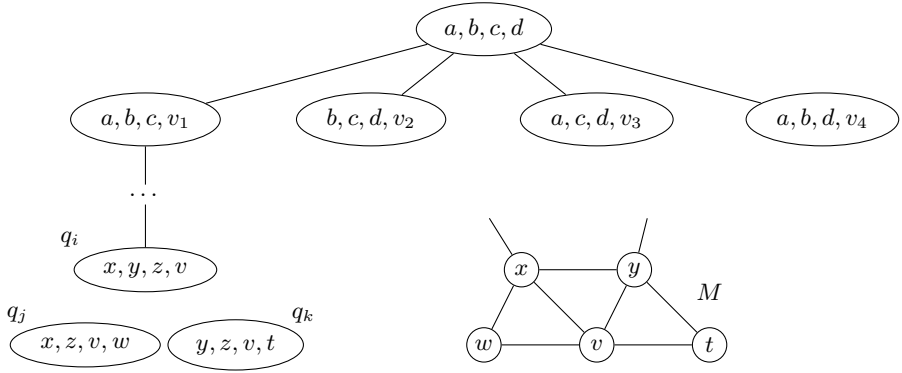

\end{proof}
The proofs of Theorems \ref{theo:toug-C6} and \ref{theo:toug-C7} rely on the fact that if the toughness of
an induced subgraph is less than 1 then the toughness of the graph is less than 1.

\begin{theorem}\label{theo:toug-C6}
Let $G\in C_6$. Then $\tau(G) < 1$.
\end{theorem}
\begin{proof}
In order to prove this result it is sufficient to show a separator of cardinality $c$ whose removal of the graph 
produces  at least $c+1$ components.

Let $q_i$ and $q_j$ be vertices of the clique-tree $T$ of $G$  with degree $4$.
Consider the neat path $P_{i,j}= \langle q_i, q_1,  \dots, q_p, q_j \rangle$. 
If $q_i$ and $q_j$ are adjacent it is immediate that 
$${{|Q_i \cup Q_j|}\over{\omega(G - (Q_i \cup Q_j))}} = {{5}\over{6}} <1.$$

Consider now $T'$  the subtree of $T$   composed by
the neat path $\langle q_i, q_1, \ldots q_p,q_j \rangle$, $p \geq 1$, and all their adjacent vertices  in $T$.
By definition,  $d(q_1)=3, \ldots, d(q_p)=3$.
Hence, tree $T'$
has $p+2$ internal vertices and $p+6$ leaves.
Let   $G'$ be the subgraph of $G$ represented by $T'$
and  $S=Q_i\cup Q_1 \cup \ldots \cup Q_p \cup Q_j$.
Observe that $|Q_1 \setminus Q_i | = 1$, $\ldots$, $|Q_j \setminus Q_p | = 1$.
Then $S$ has $4+p+1$ elements.
Each leaf of $T'$ corresponds to a maximal clique of $G'$ that has a simplicial vertex in $G'$.
When removing $S$ of $G'$ these simplicial vertices become components.
There are $3+p+3$ leaves.
So, 
$\tau(G') = \frac{p+5}{p+6}$ and $\tau(G) \leq \tau(G') <1$.
\qed
\end{proof}

\begin{theorem}\label{theo:toug-C7}
Let $G\in C_7$.
Then  $\tau(G) < 1$.
\end{theorem}

\begin{proof}
In order to prove that $\tau(G)< 1$ it is sufficient to show a separator of cardinality $c$ whose removal of the graph 
produces  at least $c+1$ components.

Let $q_i$ and $q_j$ be vertices of the clique-tree $T$ of $G$  with degree $4$.
Consider the neat path $P_{i,j}= \langle q_i, q_1,  \dots, q_p, q_j \rangle$. 
Consider $T'$  the subtree of $T$   composed by the neat path 
$P_{i,j}$ and the adjacent vertices of the vertices of the path.
As $G \in C_7$, the path $P= \langle q_1,  \dots, q_p \rangle$ is a slim path.

Three cases must be considered:
\begin{enumerate}
  \item  $|Q_i \cap Q_j|=2$ with $p \geq 2$.
 
The tree $T'$ has at least $6$ leaves: $3$ are adjacent to $q_i$ and $3$ are adjacent to $q_j$.
Let   $G'$ be the subgraph of $G$ represented by $T'$
and  $S=Q_i\cup  Q_j$; $|S| = 6$.
After the removal of $S$, it remains in $G'$ at least six components that are 
the simplicial vertices of the maximal cliques that correspond to 
 the leaves of $T'$  and 
one more establish by 
$(Q_1 \cup \ldots \cup Q_p)\setminus S$.
As $p\geq 2$,  this last component has also at least one vertex.
So, $\tau(G) < 1$.

  \item $|Q_i \cap Q_j|=  1$ with $p \geq 3$ and $G$ contains at least one maximal clique $Q_k$ 
such that $(Q_i \cup Q_j) \supset Q_k$, $d(q_k) =3$.

The proof is analogous to the proof of case 1. 
The tree $T'$ has at least $7$ leaves: $3$ are adjacent to $q_i$,
$3$ are adjacent to $q_j$ and one is adjacent to $q_k$.
Let   $G'$ be the subgraph of $G$ represented by $T'$
and  $S=Q_i\cup  Q_j$; $|S| = 7$.
After the removal of $S$, it remains in $G'$ at least  seven components that are 
the simplicial vertices of the maximal cliques that correspond to 
 the leaves of $T'$  and 
one more establish by 
$(Q_1 \cup \ldots \cup Q_p)\setminus S$.
As $p\geq 3$, this last component has also at least one vertex.
So, $\tau(G) < 1$.

  \item $|Q_i \cap Q_j|= 0$ with $p \geq 4$ and $G$ contains at least two maximal cliques $Q_k$ and $Q_\ell$ 
such that $(Q_i \cup Q_j) \supset Q_k$, $(Q_i \cup Q_j) \supset Q_\ell $, $d(q_k) =d(q_\ell) = 3$.

The reasoning is similar to the case 2.
\qed
\end{enumerate}
\end{proof}

Figure  \ref{fig:teoc7} presents an illustration of case 1 of Theorem  \ref{theo:toug-C7} with $n\geq 13$.
Consider $S=\{a,b,c,d,x,y\}$.  It is immediate to see that there 
are the following connected components in $G - S$:
$\{\ldots, 1\}$, $\{\ldots, 2\}$, $\{\ldots, 3\}$, $\{\ldots, e\}$,
$\{\ldots, 4\}$, $\{\ldots, 5\}$ and $\{\ldots,6\}$.

\medskip

\begin{figure}[h!]
\centering
{\small

\begin{tikzpicture}  [scale=.9]

\tikzstyle{block}= [ellipse,draw,minimum width=40pt]
\node[block] (1a)    at (8.5,0) {$a, b, c, d$};
\node[block] (2a)    at (3.5,-1.5) {$a, b, c, e$};
\node[block] (2b)    at (6,-1.5) {$a, b,d, 1$};
\node[block] (2c)    at (8.5,-1.5) {$a, c, d, 2$};
\node[block] (2d)    at (11,-1.5) {$b, c, d, 3$};
\node[block] (3a)    at (1.8,-3) {$a,b,....$};
\node[block] (4a)    at (0,-4.5) {$a, b, x, y$};
\node[block] (5a)    at (0,-6) {$a, b, y, 4$};
\node[block] (5b)    at (2.4,-6) {$a, x,y, 5$};
\node[block] (5c)    at (4.9,-6) {$b, x, y, 6$};
\foreach \from/\to in {1a/2a,1a/2b,1a/2c,1a/2d,
4a/5a, 4a/5b, 4a/5c}
\draw (\from) -- (\to);
\foreach \from/\to in {2a/3a, 3a/4a}
\draw[dashed](\from) -- (\to);
\end{tikzpicture}
}
\caption{Clique-tree of case 1   of Theorem \ref{theo:toug-C7}} 
\label{fig:teoc7}
\end{figure}
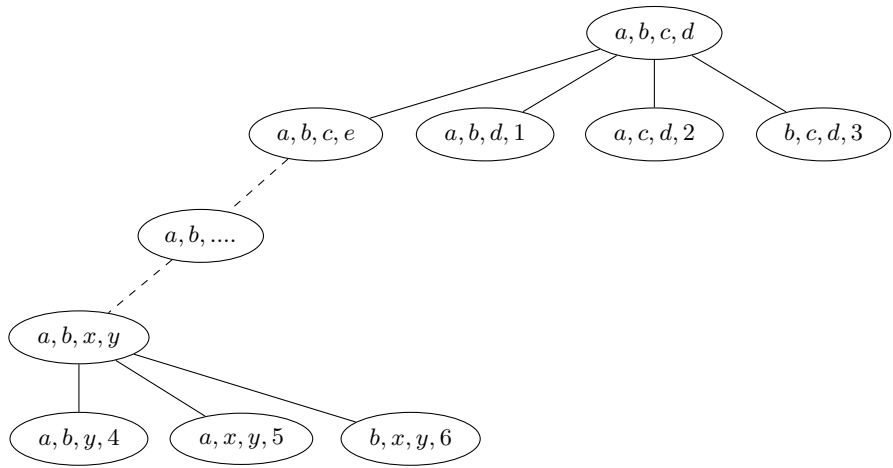


\section{Main results -- hamiltonicity}

In this section,  results on the hamiltonicity of the subclasses defined in Section 5 are presented. 

\begin{theorem}
Let $G$ be a RAN that belongs to $C_0, C_1$ on 8  vertices, $C_2,C_3, C_4$ or $ C_5$.
Then $G$ is Hamiltonian.
\end{theorem}
\begin{proof}
Let be $G \in C_0 \cup C_2 \cup C_3 \cup C_4$.
By Theorems \ref{theo:toug-C0}, \ref{theo:toug-C2}, \ref{theo:toug-C3} and \ref{theo:toug-C4}, $\tau (G) >1$ and, 
by Theorem \ref{theo:Bohme}, $G$  is Hamiltonian.

By Theorem \ref{theo:toug-C1}, $G\in C_1$ on 8  vertices is such that $\tau (G) =1$. By inspection, $G$ is Hamiltonian.

The proof of Theorem \ref{theo:toug-C5} builds a spanning \emph{mop}
of the graph $G\in C_5$; so, by Theorem \ref{theo:ktreehamileq}, $G$ is also Hamiltonian. 
\qed
\end{proof}

\begin{theorem}
Let $G$ be a RAN that belongs to $C_1$ with $n\geq 11$ vertices, $C_6$ or $C_7$.
Then $G$ is non-Hamiltonian.
\end{theorem}
\begin{proof}
Let be $G \in C_1$ with $n\geq 11$ vertices, $C_6$ or $C_7$.
By Theorems \ref{theo:toug-C1}, \ref{theo:toug-C6} and \ref{theo:toug-C7}, $\tau (G) < 1$. 
Then, by Theorem \ref{theo:Chavtal}, $G$  is non-Hamiltonian. \qed
\end{proof}


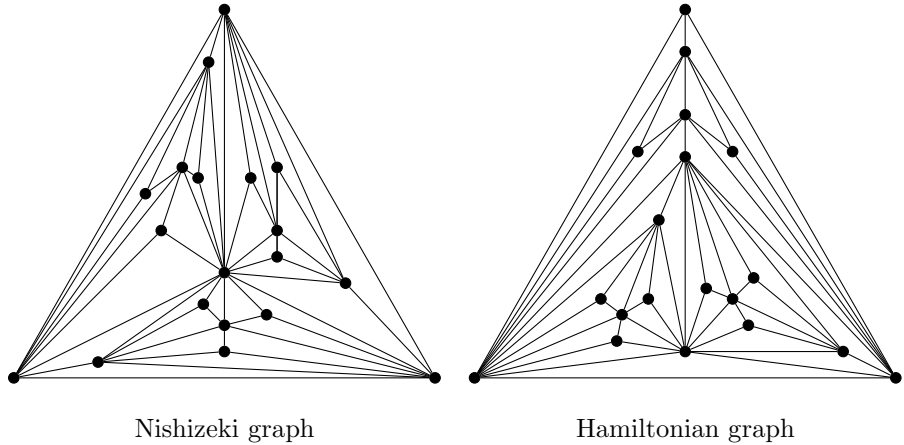
\begin{figure}[!h]
\tikzset{
    graphnode/.style={
      draw,fill,
      circle,
      minimum size=1.4mm,
      inner sep=0
    }
}
\begin{minipage}{6.5cm}
\begin{center}
\begin{tikzpicture}[every node/.style={graphnode},scale=.7]
  \node (a)at (4,7) {};
  \node (b) at (0,0) {};
  \node (c) at (8,0) {};
   \node (d) at (4,2) {};
   \node (1) at (3.7,6) {};
  \node (2) at (3.2,4) {};
  \node(3) at (2.5,3.5) {};
 \node(4) at (3.5,3.8) {};
  \node(5) at (2.8,2.8) {};
   \node(8) at (4.5,3.8) {};
   \node (9) at (5,4) {};
  \node(7) at (5,2.8) {};
 \node(6) at (6.3,1.8) {};
  \node(10) at (5,2.3) {};
\node(11) at (1.6,0.3) {};
   \node (12) at (4,1) {};
 \node (13) at (4.8,1.2) {};
 \node(15) at (3.6,1.4) {};
   \node (14) at (4,0.5) {};
\foreach \from/\to in {a/b,a/c, a/d,b/d,c/d,b/c,1/a,1/b,1/d,1/2,2/b,2/d,4/1,4/2,4/d,3/1,3/b,3/2,5/b,5/2,5/d,8/a,8/d,8/7,7/a,7/d,9/a,9/7,6/a,6/9,6/7,6/c,6/d,10/7,10/9,10/6,10/d,11/b,11/c,11/d,15/11,15/d,12/d,12/15,12/11,12/c,13/d,13/12/,13/c,14/12,14/11,14/c}
  \draw (\from) -- (\to);
  \node at (4,-1) [draw=none,fill=none] {Nishizeki graph};
\end{tikzpicture}
\end{center}
\end{minipage}\hfill
\begin{minipage}{6.5cm}
\vskip.25cm
\begin{tikzpicture}[every node/.style={graphnode},scale=.7]
  \node (b) at (0,0) {};
  \node (c) at (8,0) {};
   \node (1) at (4,0.5) {};
   \node (h) at (4,4.2) {};
   \node (d) at (4,5) {};
   \node (a) at (4,6.2) {};
   \node (e) at (4,7) {};
    \node (f) at (3.1,4.3) {};
   \node (g) at (4.9,4.3) {};
   \node (2) at (3.5,3) {};
   \node (3) at (2.8,1.2) {};
        \node (y) at (2.4,1.5) {}; \node (z) at (3.3,1.5) {}; \node (x) at (2.7,0.7) {};
   \node (4) at (7,0.5) {};   
   \node (5) at (4.9,1.5) {};     
          \node (v) at (4.4,1.7) {}; \node (w) at (5.3,1.9) {};  \node (t) at (5.2,1) {}; 
\foreach \from/\to in {a/b,a/c,a/d,a/e,b/c,b/d,b/1,c/d,c/1,b/e,c/e,b/h,c/h,1/h,d/h,a/f,d/f,b/f,a/g,d/g,c/g,b/2,h/2,1/2,2/3,b/3,1/3,b/y,2/y,3/y,2/z,3/z,1/z,b/x,3/x,1/x,c/4,1/4,h/4,h/5,1/5,4/5,h/v,1/v,5/v,h/w,5/w,4/w,5/t,1/t,4/t} 
  \draw (\from) -- (\to);
  \node at (4,-1) [draw=none,fill=none] {Hamiltonian graph};
\end{tikzpicture}
\end{minipage}
\vskip-0,8cm
\caption{Graphs belonging to $C_8$} 
\label{fig:c8}
\end{figure}

%

\section{Conclusions}

We have established a partition of the class
of \emph{RAN}s in 8 subclasses and we were able to develop strong results in relation 
to toughness and hamiltonicity for subclasses $C_1$ to $C_7$. 
It remains to be studied the behavior of graphs belonging to $C_8$.

We conjecture that all graphs belonging to $C_8$ have toughness equal to 1.
However, the reasoning applied to the proofs of previous theorems does not apply to this class
and the structure of the clique-trees does not provide new insights.

With regard to hamiltonicity, some results are already known showing that $C_8$ contains both Hamiltonian and non-Hamiltonian graphs.
For instance, the graph presented by Nishizeki \cite{Ni80} (shown in Figure \ref{fig:c8}) belongs to $C_8$ and it is non-Hamiltonian.
Let us consider its clique-tree $T_N$.
It  has  three vertices, $q_1$, $q_2$ and $q_3$ with degree 4; 
the  paths joining them are  slim paths and $|Q_1\cap Q_2\cap Q_3| =  1$. 
Also in Figure \ref{fig:c8}, 
we present another graph  $G\in C_8$ on the same number of vertices and with
the clique-tree  isomorphic  to $T_N$;
however $Q_1\cap Q_2 \cap Q_3$ is empty and  $G$ is a Hamiltonian graph. 
This observation leads us to conjecture  that the non-empty intersection of the maximal cliques
 has a close relation to non-hamiltonicity.


\end{document}